\documentclass[a4pape, reqno,12pt,oneside]{amsart}

\usepackage{latexsym}
\usepackage{amssymb}
\usepackage{amsmath}
\usepackage{graphicx}
\usepackage{wrapfig} 
\usepackage{amsthm}
\usepackage{float}
\usepackage{epsfig}
\usepackage{multirow}
\usepackage[toc,page]{appendix}
\usepackage[footnotesize]{caption}
\RequirePackage{color}
\usepackage{cite}

\usepackage{tikz}
\usetikzlibrary{calc,matrix,arrows,snakes,backgrounds,positioning,shapes,automata,decorations}
\usetikzlibrary{trees,positioning}
\usetikzlibrary{decorations.markings}
\tikzstyle directed=[postaction={decorate,decoration={markings, mark=at position .65 with {\arrow{stealth}}}}]
\tikzstyle reverse directed=[postaction={decorate,decoration={markings,
   mark=at position .65 with {\arrowreversed{stealth};}}}]

\numberwithin{figure}{section}
\usepackage{subfigure}

\usepackage{amscd,enumerate}
\usepackage{amsfonts}
\usepackage{enumitem} 

\input xy
\xyoption{all}

\usepackage{multicol}
\usepackage{hyperref}
\hypersetup{ colorlinks,
linkcolor=blue,
filecolor=green,
urlcolor=blue,
citecolor=blue }

\newtheorem{lemma}{Lemma}[section]
\newtheorem{corollary}[lemma]{Corollary}
\newtheorem{theorem}[lemma]{Theorem}
\newtheorem{proposition}[lemma]{Proposition}

\newtheorem{definition}[lemma]{Definition}

\newtheorem{example}[lemma]{Example}

\usepackage{tikz}
\usetikzlibrary{arrows,snakes,backgrounds,trees}

\usepackage{stmaryrd}

\def\T+{{\mathbb T_d^+}}

\def\A{\mathcal{A}}

\def\conv {\mathop {\rm conv }\nolimits}

\begin{document}

\subjclass[2010]{17D92, 46C99, 60J10}
\keywords{Genetic Algebra, Evolution Algebra, Hilbert Space, Markov Chain}

\title[Hilbert Evolution Algebras]{Hilbert Evolution Algebras and its connection with discrete-time Markov Chains}

\author[Sebastian J. Vidal]{Sebastian J. Vidal}
\address{Sebastian J. Vidal: Departamento de Matem\'atica, Facultad de Ingenier\'ia, Universidad Nacional de la Patagonia ``San Juan Bosco'', Km 4, CP 9000, Comodoro Rivadavia, Chubut, Argentina.}
\email{sebastianvidal79@gmail.com}

\author[Paula Cadavid]{Paula Cadavid}
\address{Paula Cadavid: Centro de Matemática, Computação e Cognição, Universidade Federal do ABC, Avenida dos Estados, 5001-Bangu-Santo Andr\'e - SP, Brazil.}
\email{pacadavid@gmail.com}

\author[Pablo M. Rodriguez ]{Pablo M. Rodriguez}
\address{Pablo M. Rodriguez: Centro de Ci\^encias Exatas e da Natureza, Universidade Federal de Pernambuco, Av. Prof. Moraes Rego, 1235 - Cidade Universit\'aria - Recife - PE, Brazil.}
\email{pablo@de.ufpe.br}

\begin{abstract}
Evolution algebras are non-associative algebras. In this work we provide an extension of this class of algebras, in the context of Hilbert spaces and  illustrate the applicability of our approach by discussing a connection with discrete-time Markov chains with infinite countable state space. 
\end{abstract}

\maketitle


\section{Introduction}

In this paper we contribute with the Theory of Evolution Algebras, which is developed around a special class of genetic algebras. At the beginning, the notion of evolution algebra was formulated in \cite{tv} as an algebraic way to mimic the self-reproduction of alleles in non-Mendelian genetics. Fortunately, like many objects in Mathematics, this concept proved to be very flexible for the comparison with concepts from different fields. The best reference to start studying the subject is the seminal work of Tian, \cite{tian}, where the author, after the formulation of basic properties for these algebras, explores an interesting correspondence between them and the theory of discrete-time Markov chains. In the same reference the reader may find a summary of possible connections with other fields like graph theory, group theory, statistical physics, and others. An evolution algebra is defined as follows.

\begin{definition}\label{def:evolalgTian} Let  $\mathbb{K}$  be a field and 
let $\A:=(\A,\cdot\,)$ be a $\mathbb{K}$-algebra. We say that $\A$ is an evolution algebra if it admits a basis $S:=\{e_i\}_{i\in\Lambda}$, such that 
\begin{eqnarray}
e_i \cdot e_i &=&\displaystyle \sum_{k\in\Lambda} c_{ki} e_k, \text{ for }i\in \Lambda,\label{eq:ea01}\\[.2cm]
e_i \cdot e_j &=& 0, \text{ for }i,j\in \Lambda \text{ such that }i\neq j.\label{eq:ea02}
\end{eqnarray} 
\end{definition}

The scalars $c_{ki}\in \mathbb{K}$ are called the structure constants of $\mathcal{A}$ relative to $S$. A basis $S$ satisfying \eqref{eq:ea02} is called natural basis of $\mathcal{A}$. We emphasize that in the definition above  basis means Hamel basis; i.e., a maximal linear independent subset. It implies that for a fixed $i \in \Lambda$ only a finite number of constants $c_{ki}$ are non-zero. 

Currently, there is a wide literature about this issue and its consequences. Here we mention some of the recent works, and we refer the reader to the references therein for a deeper study of the theory. In \cite{PMP2,PMPY,camacho/gomez/omirov/turdibaev/2013,casado/molina/velasco/2016,Casas/Ladra/2014} the reader may find a survey of properties and results for general evolution algebras; the works in \cite{Elduque/Labra/2015,PMP,PMPT,reis} are devoted to the connection between evolution algebras and graphs together with some related properties; and in \cite{Falcon/Falcon/Nunez/2017,Labra/Ladra/Rozikov/2014} one may see a good review of results with relevance in genetics and other applications.

We are interested in providing a generalization of Definition \ref{def:evolalgTian} which be able to deal with infinite-dimensional spaces and, at the same time, to include an application not covered by it. Let us start with our motivation. If $\A$ is an evolution algebra such that $c_{ki}\in[0,1]$, for any $i,k\in \Lambda$, and $\sum_{k\in \Lambda} c_{ki}=1$, for any $i\in \Lambda$, then $\A$ is called a Markov evolution algebra. The name is due to a correspondence between evolution algebras and discrete-time Markov chains given in \cite{tian}. To see the connection, let us remember some basic notation for Markov chains. Consider a probability space $(\Omega,\mathcal{F},\mathbb{P})$; i.e., $\Omega$ is an arbitrary non-empty set, $\mathcal{F}$ is a $\sigma$-field of subsets of $\Omega$ and $\mathbb{P}$ is a probability measure on $\mathcal{F}$. A sequence of random variables $\{X_{n}\}_{n\geq 0}$ living in this probability space and taking values in $\mathcal{X}:=\{x_i\}_{i\in\Lambda}$, where $\Lambda$ is a countable set of indices, is called a discrete-time Markov chain if it satisfies the Markov's property; namely,  
$$\mathbb{P}(X_{n+1}=x_j|X_{0}=x_{i_0},X_{1}=x_{i_1},\ldots,X_{n-1}=x_{i_{n-1}},X_{n}=x_i)=\mathbb{P}(X_{n+1}=x_j|X_{n}=x_i)=:p_{ij},$$ for any $n\geq 1$, and for any subset $\{x_{i_0},x_{i_1},\ldots,x_{i_{n-1}},x_i,x_j\}\subset \mathcal{X}$ with $\{i_0,i_1,\ldots,i_{n-1},i,j\}\subset \Lambda$. The values $p_{ij}$ are called the transition probabilities of the Markov chain and do not depend on $n$; i.e., $\{X_{n}\}_{n\geq 0}$ is an homogeneous Markov chain. In \cite[Chapter 4]{tian} it is defined an evolution algebra $\A$ with a natural basis $S:=\{e_i\}_{i\in\Lambda}$ in such a way that each state of the Markov chain is in correspondence with each generator of $S$, and $c_{ki}=p_{ik}$, for any $i,k\in\Lambda$.

As far as we known \cite[Chapter 4]{tian} was the first in proposing the interplay between evolution algebras and Markov chains. In such a work many well-known results coming from Markov chains were stated in the language of Markov evolution algebras. We point out that this is an interesting connection which deserves to be explored because it represents a new framework to describe random phenomena; i.e. through techniques of non-associative algebras. However, we have to take care when dealing with the connection of these mathematical objects. Although \cite[Theorem 16, page 54]{tian} claims, using the correspondence mentioned above, that for any homogeneous Markov chain there is an evolution algebra whose structure constants are transition probabilities, and whose generator set is the state space of the Markov chain, this is not totally true whether the state space has infinitely many elements. 

\begin{example}\label{ex:mc} [A Discrete-time Markov chain that does not determine an evolution algebra according to Definition \ref{def:evolalgTian}] Let $\{X_n\}_{n\geq 0}$ be a Markov chain with state space given by $\mathcal{X}=\mathbb{N}\cup\{0\}$ and transition probabilities given by $p_{0i}=p_i>0$, for any $i\in\mathbb{N}$, where $\sum_{i=1}^{\infty}p_i=1$, and $p_{i (i-1)}=1$ for any $i\in\mathbb{N}$. See Figure \ref{fig:mc} for an illustration of the transitions of this Markov chain. In words, from any state $i\neq 0$ the process ``jumps'' to state $i-1$ with probability $1$, and as soon as the process hits state $0$, it jumps to state $i$ with probability $p_i$. Notice that from state $0$ we can go to infinitely many states with positive probability. If we assume that there exists an evolution algebra whose generator set is in correspondence with the state space of this Markov chain, namely $S=\{e_i\}_{i\in \mathbb{N}\cup\{0\}}$, then, taking $c_{ki}=p_{ik}$, it should be for $i\neq 0$, $e_i^2=e_{i-1}$, while
\begin{equation*}
e_0^2 = \sum_{i\in \mathbb{N}} p_i e_i,
\end{equation*}
with $p_i> 0$ for any $i\in \mathbb{N}$. But this is a contradiction because according to Definition \ref{def:evolalgTian}, since $S$ is a Hamel basis, the numbers $c_{ki}$ can be non-zero only for a finite number of $j$'s.
\end{example}

\begin{figure}
    \centering
    \begin{tikzpicture}[scale=0.8, every node/.style={scale=0.8}]

\draw [thick] (0,0) circle (7pt);
\draw (0,0) node[font=\footnotesize] {$0$};
\draw [thick] (2,0) circle (7pt);
\draw (2,0) node[font=\footnotesize] {$1$};
\draw [thick] (4,0) circle (7pt);
\draw (4,0) node[font=\footnotesize] {$2$};
\draw [thick] (6,0) circle (7pt);
\draw (6,0) node[font=\footnotesize] {$3$};
\draw [thick] (8,0) circle (7pt);
\draw (8,0) node[font=\footnotesize] {$4$};


\draw [thick, reverse directed] (0,0.25) to [bend left=45] (2,0.25);
\draw [thick, reverse directed] (2,-0.25) to [bend left=45] (0,-0.25);
\draw [thick, reverse directed] (4,-0.25) to [bend left=60] (0,-0.25);
\draw [thick, reverse directed] (6,-0.25) to [bend left=75] (0,-0.25);
\draw [thick, reverse directed] (8,-0.25) to [bend left=90] (0,-0.25);
\draw [dashed, gray, reverse directed] (10,-1.25) to [bend left=90] (0,-0.25);
\draw [thick, reverse directed] (2,0.25) to [bend left=45] (4,0.25);
\draw [thick, reverse directed] (4,0.25) to [bend left=45] (6,0.25);
\draw [thick, reverse directed] (6,0.25) to [bend left=45] (8,0.25);
\draw [<-, dashed, gray] (8,0.25) to [bend left=45] (10,0.45);


\draw (1,1) node[font=\footnotesize] {$1$};
\draw (1.6,-0.8) node[font=\footnotesize] {$p_1$};
\draw (3.3,-1.3) node[font=\footnotesize] {$p_2$};
\draw (5.4,-1.6) node[font=\footnotesize] {$p_3$};
\draw (7.3,-2.1) node[font=\footnotesize] {$p_4$};
\draw (3,1) node[font=\footnotesize] {$1$};
\draw (5,1) node[font=\footnotesize] {$1$};
\draw (7,1) node[font=\footnotesize] {$1$};
\draw (9,0) node[font=\footnotesize] {$\cdots$};

\end{tikzpicture}
    \caption{Graphical representation of the Markov chain of Example \ref{ex:mc}. States are represented by vertices of the directed graph. Directed edges represent possible transitions between states while weights in the edges represent the respective transition probabilities.}
    \label{fig:mc}
\end{figure}
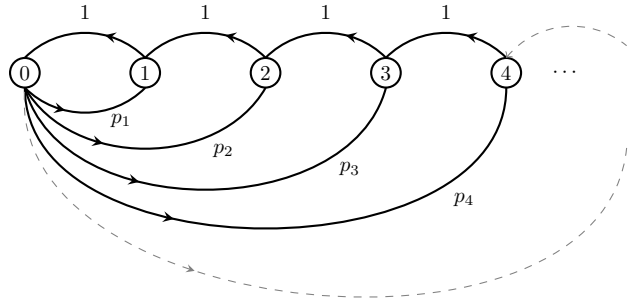


The previous example suggest that although many connections with other fields have been suggested in literature, still some gaps exist whether one consider applications involving infinite-dimensional spaces. This is because in the original definition of Tian \cite{tian} the basis is implicitly assumed to be a Hamel basis. Thus, the sum \eqref{eq:ea01} can have only a finite number of nonzero terms. Following this work, in \cite{casado/molina/velasco/2016} the authors consider infinite-dimensional evolution algebras but still only with finite sums. In order to allow an infinite number of nonzero terms in the series, we need to consider other structures to give a meaning to the sum \eqref{eq:ea01}. The usual way to do this is through Functional Analysis, by introducing topologies and different notions of convergence. One approach to do this was taken in \cite{Mellon/Velasco}. The novelty of such a work is the concept of Banach evolution algebras capable to deal with infinite-dimensional algebras whose natural basis are uncountable. However, the case of evolution algebras with an infinite countable natural basis is not covered by their definition. With the motivation of fulfill this gap we propose a different approach and we work with Hilbert spaces, which leads us to consider other kind of basis; namely, Schauder basis. In other words, we propose an extension of Definition \ref{def:evolalgTian} by providing an evolution algebra structure in a given Hilbert space. We call that new structure a Hilbert evolution algebra, and after stating some basic properties, we illustrate its applicability to the connection with discrete-time Markov chains.

The rest of the paper is subdivided into two sections. In Section 2 we introduce the concept of Hilbert evolution algebra, we define its associated evolution operator, and we discuss a condition under which this operator is continuous. In Section 3 we include our application to the connection between these objects and discrete-time Markov chains.


\section{Hilbert Evolution Algebras}
We start with some definitions and notation. Let  $V$ be a vector space over  $\mathbb{K}$, where  $\mathbb{K}$ is $\mathbb{R}$  or $\mathbb{C}$, with inner product $ \langle \, , \, \rangle$. A subset $\{e_k\}_{k\in\Lambda}\subset V$, where $\Lambda$  is a countable set, is a Schauder Basis of $V$ if any $v\in V$ has a unique representation
	\begin{equation*} \label{rep}
	v=\sum_{k\in\Lambda} v_k e_k, \,\,\text{  where } v_{k}  \in  \mathbb{K}.
	\end{equation*}
We say that $V$ is a Hilbert space if it is also a complete metric space with respect to the distance function induced by the inner product.  A subset $\{e_k\}_{k\in\Lambda}\subset V$ is an orthonormal basis if every $v\in V$ can be expressed as
\begin{equation*}
    v=\sum_{k\in\Lambda}\langle v,e_k\rangle e_k.
\end{equation*}
On the other hand, we  say that $V$ is separable if it has a countable dense subset. In this case, any orthonormal basis is countable. The Gram-Schmidt orthonormalization process proves that every separable Hilbert space has an orthonormal basis. We highlight  that if $V$ is finite-dimensional, the notion of Schauder basis coincides with that of Hamel basis. 

We shall define an evolution algebra structure in a Hilbert space $\A$. In order to do it two questions should be considered. The first one is that we would like to define a product in $\A$ satisfying relations \eqref{eq:ea01} and  \eqref{eq:ea02}. The problem with that is the convergence of the series involved in the definition of such a new product; that is, if $v, w\in \A$ then $v\cdot w$ may not be in $\A$. Specifically, if we write $v=\sum_{k\in\Lambda}v_k e_k$ and $w=\sum_{k\in\Lambda}w_k e_k$  where  $\{e_k\}_{k\in\Lambda}$ is an orthonormal basis then, using \eqref{eq:ea01} and \eqref{eq:ea02} and extending  by linearity, we must have
\begin{equation*}
v\cdot w=\sum_{k\in\Lambda}\biggl(\sum_{i\in\Lambda}v_iw_ic_{ki}\biggr) e_k.
\end{equation*}
However, the series can be non convergent in $\A$. To solve this problem we will work with \textit{separable} Hilbert spaces and appeal to a well-known result of Hilbert spaces theory. 
\begin{proposition}\label{prop:convergence of series}\cite[Theorem 8.3.1]{Heil}
Let $\A$ be a Hilbert space and let $\{e_i\}_{i\in\mathbb{N}}$ be an orthonormal subset. The series $\sum_{k=1}^{\infty}c_ke_k$ is convergent if, and only if, the numerical series $\sum_{k=1}^{\infty} |c_k|^2$ is convergent.
\end{proposition}
Thus, given a separable Hilbert space $\A$, we want to define the product algebra for elements $v=\sum_{k=1}^{\infty}v_k e_k$ and $w=\sum_{k=1}^{\infty}w_k e_k$ satisfying 
\begin{equation}\label{eq:vw convergent}\displaystyle
\sum_{k=1}^{\infty}\biggl|\sum_{i=1}^{\infty}v_iw_ic_{ki}\biggr|^2 <\infty,
\end{equation}
for an orthonormal basis $\{e_k\}_{k\in\mathbb{N}}$. In this case, the product $\cdot:\A\times\A\longrightarrow \A$ can be defined in the basis $\{e_k\}_{k\in\mathbb{N}}$ and extended by linearity. 

The second issue to consider for a general definition of evolution algebra in a Hilbert space is the compatibility between the involved structures. Note that under the considerations described above we can introduce the left multiplication operator
\begin{equation} \label{eq:Lv}
\begin{array}{rrcl}
L_{v}:& \A & \longrightarrow&  \A \\
 & w & \longmapsto & L_{v}(w):= v\cdot w,
\end {array}
\end{equation}
for any $v\in\A$. So we shall require for the continuity of left multiplication operators whenever is possible to define the product algebra. After the previous considerations we are able to introduce our definition.
\begin{definition}\label{def:evolalg} Let $\A=(\A,\langle\cdot, \cdot\rangle)$ be a real or complex separable Hilbert space  which is provided with an algebra structure by the product $\cdot:\A\times\A\rightarrow \A$. We say that $\A:=(\A,\langle\cdot, \cdot\rangle,\cdot \,)$ is a separable Hilbert evolution algebra if it satisfies the following conditions:
\begin{enumerate}[label=(\roman*)]
\item There exists an orthonormal basis $\{e_i\}_{i\in\mathbb{N}}$ and scalars $\{c_{ki}\}_{i,k\in\mathbb{N}}$, such that 
\begin{equation}\label{eq:ea03}
e_i \cdot e_i =\displaystyle \sum_{k=1}^{\infty} c_{ki} e_k
\end{equation}
and
\begin{equation}\label{eq:ea04}
e_i \cdot e_j=0, \text{ if }i\neq j,
\end{equation} 
for any $i,j\in \mathbb{N}$. 
\smallskip
\item For any $v\in\A$, the left multiplications $L_v$ defined by \eqref{eq:Lv} are continuous in the metric topology induced by the inner product; i.e., there exists constants $M_v>0$ such that
\begin{equation}\label{eq:continuity condition}
\|L_v(w)\|\leq M_v \|w\|,\text{ for all } w\in \A.
\end{equation}
\end{enumerate}
\end{definition}
A basis satisfying condition (i) will be called {\it orthonormal natural basis}.  In the sequel we will work only with separable Hilbert spaces, so we omit the word separable and talk about Hilbert evolution algebras. As the evolution algebras in the sense of Definition \ref{def:evolalgTian}, the Hilbert evolution algebras are commutative and are, in general, non associative and without an unitary element. Also, it is not difficult to see that for any finite-dimensional evolution algebra it is possible to define a norm such that the algebra becomes an Hilbert  evolution algebra. For more details see \cite[Section 3.3]{tian}. Let us also point out that while checking \eqref{eq:continuity condition}, we are also checking that the product algebra is well defined; that is, if \eqref{eq:continuity condition} holds then $\|L_v(w)\|=\|v\cdot w\|<\infty$, which written explicitly in any orthonormal natural basis is equivalent to \eqref{eq:vw convergent}. 

Note that if $V$ is a finite-dimensional vector space and $\{e_i\}_{i\in\Lambda}$ is a basis for $V$,  then it is always possible to give an evolution algebra structure to $V$ by defining a product satisfying the equations \eqref{eq:ea01} and \eqref{eq:ea02}, for any finite subset of scalars $\{c_{ki}\}_{i,k\in\Lambda}$. However, as we prove in the next result, in the infinite-dimensional Hilbert case the sequence of scalars must satisfy an additional condition.

\begin{proposition} 
Every separable Hilbert space $\A$ admits a Hilbert evolution algebra structure. Moreover, if $\{e_i\}_{i\in\mathbb{N}}$ is an orthonormal basis, then every sequence $\{c_{ki}\}_{i,k\in\mathbb{N}}$ such that
\begin{equation}\label{eq:condition HEA01}
\sum_{k=1}^{\infty}\sum_{i=1}^{\infty}|v_i c_{ki}|^2 <\infty,
\end{equation}
for any $v=\sum_{i=1}^{\infty}v_i e_i\in \A$, defines a Hilbert evolution algebra structure in $\A$. 
\end{proposition}
\begin{proof}Let $v=\sum_{i=1}^{\infty}v_i e_i$ and note that always exist numbers satisfying \eqref{eq:condition HEA01}. Indeed, it is  possible to choose numbers $\{c_{ki}\}_{i,k\in\mathbb{N}}$ such that 
\begin{equation*}
\sup\biggl\{\sum_{k=1}^{\infty}|c_{ki}|^2 \,:\, i\in\mathbb{N} \biggr\}<\infty,
\end{equation*}from which \eqref{eq:condition HEA01} follows immediately, because $\|v\|^2=\sum_{i=1}^{\infty}|v_i|^2 <+\infty$, since $v\in\A$. Now, let $w=\sum_{i=1}^{\infty}w_i e_i$, and consider the formal series
\begin{equation}\label{eq:Lv series}
L_v (w)=\sum_{k=1}^{\infty}\biggl(\sum_{i=1}^{\infty}v_iw_ic_{ki}\biggr) e_k.
\end{equation}
Thus, to prove the proposition we must analyze the convergence of 
\begin{equation*}
\|L_v(w)\|^2=\sum_{k=1}^{\infty}\biggl|\sum_{i=1}^{\infty}v_iw_ic_{ki}\biggr|^2.   
\end{equation*}
By equation \eqref{eq:condition HEA01} we have that $\sum_{i=1}^{\infty}|v_i c_{ki}|^2 <\infty$, for any $k\in\mathbb{N}$.
On the other hand $w\in\A$ implies that $\|w\|^2=\sum_{i=1}^{\infty}|w_i|^2<\infty$. 
Then we can use the Cauchy-Schwartz inequality to obtain
\begin{equation*}
\biggl|\sum_{i=1}^{\infty}w_iv_ic_{ki}\biggr|^2\leq \left(\sum_{i=1}^{\infty}|w_i|^2\right)\left(\sum_{i=1}^{\infty}|v_ic_{ki}|^2\right)=\|w\|^2\sum_{i=1}^{\infty}|v_ic_{ki}|^2,       
\end{equation*}
for every $k\in\mathbb{N}$. Hence
\begin{equation*}
\sum_{k=1}^{\infty}\biggl|\sum_{i=1}^{\infty}w_iv_ic_{ki}\biggr|^2   \leq\|w\|^2\sum_{k=1}^{\infty}\sum_{i=1}^{\infty}|v_ic_{ki}|^2= M_v^2 \|w\|^2,  
\end{equation*}
where by hypothesis we can define 
\begin{equation*}
M_v:=\biggl(\sum_{k=1}^{\infty}\sum_{i=1}^{\infty}|v_i c_{ki}|^2\biggr)^{1/2}.
\end{equation*}
It follows that $\|L_v(w)\|\leq M_v\|w\|$, for any $w\in \A$; i.e., the operators $L_v$ are well defined and are continuous for $v\in \A$. That is, if we define the product in the basis $\{e_i\}_{i\in\mathbb{N}}$ by the equations \eqref{eq:ea03} and  \eqref{eq:ea04} then, it is possible to extend the product by linearity to $v\cdot w$ for all $v,w\in\A$ using the equation \eqref{eq:Lv series}, and in this framework the operators $L_v$ are continuous. Therefore we have an Hilbert evolution algebra structure defined in $\A$.
\end{proof}

Based on the previous proof we see that there is an important special case to guarantee the existence of Hilbert evolution algebras.
\begin{corollary}\label{corol:manageable condition}
Let $\A$ be a separable Hilbert space. Consider an orthonormal basis $\{e_i\}_{i\in\mathbb{N}}$ and suppose that the sequences $\{c_{ki}\}_{i,k\in\mathbb{N}}$ satisfy
\begin{equation}\label{eq:condition HEA02}
K:=\sup\biggl\{\sum_{k=1}^{\infty}|c_{ki}|^2 \,:\, i\in\mathbb{N} \biggr\}<\infty.
\end{equation}
Then $\A$ admits an Hilbert evolution algebra structure where the $c_{ki}$ are the structure constants.
\end{corollary}
\begin{proof}
Just note that \eqref{eq:condition HEA02} implies \eqref{eq:condition HEA01}.
\end{proof}

In analogy to the theory of \cite[Section 3.2]{tian} in the finite dimensional case, the next step is to introduce the  evolution operator induced by the Hilbert evolution algebra. We define the evolution operator as the linear operator $C:D(C)\longrightarrow \A$ given by its values in a natural orthonormal basis,
\begin{equation*}
C(e_i):=e_i^2=\sum_{k=1}^{\infty} c_{ki}e_k ,
\end{equation*}
and with domain $D(C)\subset \A$ defined as
\begin{equation}\label{eq:domain of C}
D(C):=\Bigg\{v=\sum_{i=1}^{\infty} v_ie_i\in\A :\, 
\sum_{k=1}^{\infty} \biggl|\sum_{i=1}^{\infty} v_ic_{ki}\biggr|^2<\infty \Bigg\}
\end{equation}
It is worth noting that $D(C)$ is a vector space. This follows from the Minkowski’s inequality, which implies
\begin{equation*}
\left(\sum_{k=1}^{\infty} \biggl|\sum_{i=1}^{\infty} (\alpha v_i+w_i)c_{ki}\biggr|^2\right)^{1/2}     
\leq \left(\sum_{k=1}^{\infty} \biggl|\sum_{i=1}^{\infty}\alpha v_ic_{ki}\biggr|^2\right)^{1/2} +\left(\sum_{k=1}^{\infty}\biggl|\sum_{i=1}^{\infty}w_ic_{ki}\biggr|^2\right)^{1/2}<\infty,
\end{equation*}
for $\alpha\in\mathbb{R}$, $v=\sum_{i=1}^{\infty} v_i e_i, w=\sum_{i=1}^{\infty} w_i e_i\in D(C)$.
With this we can write
\begin{equation}\label{eq:EOp00}
	C(v):=\sum_{k=1}^{\infty} \biggl(\sum_{i=1}^{\infty} v_ic_{ki}\biggr)e_k,
\end{equation}
for any $v=\sum_{i=1}^{\infty} v_i e_i \in D(C)$.
In the general case the operator $C$ will be unbounded, thus it is important to find conditions on the structure constants to know when $C$ is closable or closed. 
This matter is left for future work.
Next we show some cases when $C$ is bounded.
	
\begin{proposition}\label{prop:EOcont}
Let $\A$ be a Hilbert evolution algebra with structure constants satisfying one of the following conditions:
\begin{enumerate}[label=(\roman*)]
    \item \label{item:condition continuity C}
    $\displaystyle\sum_{k=1}^{\infty}\sum_{i=1}^{\infty}|c_{ki}|^2<\infty.$ \vspace{0.2cm}
   	
    \item There exists $\alpha_k, \beta_i>0$, $i,k\in \mathbb{N}$ and $M_1, M_2>0$ such that
    \begin{equation}\label{eq:schur test}
    \begin{array}{l}
    \displaystyle\sum_{k=1}^{\infty}|c_{ki}|\alpha_k\leq M_1\beta_i, \quad\text{ for all } i\in \mathbb{N}, \\[1.1ex]
    \displaystyle\sum_{i=1}^{\infty}|c_{ki}|\beta_i\leq M_2\alpha_k,
    \quad\text{ for all } k\in \mathbb{N}.
    \end{array}
    \end{equation}
\end{enumerate}
Then $D(C)=\A$ and the evolution operator $C:\A\longrightarrow \A$ is bounded with
$\|C\|\leq (M_1M_2)^{1/2}.$
\end{proposition}
\begin{proof}
Suppose the first condition. Let $v=\sum_{i=1}^{\infty} v_i e_i\in\A$ and note that 
\begin{equation*}
\|C(v)\|^2
=\sum_{k=1}^{\infty} \biggl|\sum_{i=1}^{\infty} v_ic_{ki}\biggr|^2  
\leq \sum_{k=1}^{\infty}\Bigg (\sum_{i=1}^{\infty} |v_i|^2\Bigg)
\Bigg(\sum_{i=1}^{\infty} |c_{ki}|^2\Bigg) 
=\|v\|^2\sum_{k=1}^{\infty}\sum_{i=1}^{\infty} |c_{ki}|^2
<\infty,
\end{equation*}
where we use the Cauchy-Schwartz inequality.
Therefore $D(C)=\A$ and the linear operator is bounded in this case. Now let us assume that the second condition \eqref{eq:schur test} is satisfied. By a similar argument used to prove the Schur Test \cite[Section 45]{Halmos}, we have that 
\begin{equation*}
\begin{array}{rl}
\|C(v)\|^2 
&\displaystyle=\sum_{k=1}^{\infty} \biggl|\sum_{i=1}^{\infty} v_ic_{ki}\biggr|^2 \\
&\displaystyle\leq \sum_{k=1}^{\infty} \biggl|\sum_{i=1}^{\infty} |v_i| |c_{ki}|\biggr|^2 \\
&\displaystyle=\sum_{k=1}^{\infty} \biggl|\sum_{i=1}^{\infty}  \left(\sqrt{|c_{ki}|}\sqrt {\beta_i}\right) \left(\frac{\sqrt{|c_{ki}|}|v_i|}{\sqrt {\beta_i}}\right)\biggr|^2 \\
&\displaystyle\leq \sum_{k=1}^{\infty} \left(\sum_{i=1}^{\infty} |c_{ki}| \beta_i\right) \left(\sum_{i=1}^{\infty} \frac{|c_{ki}| |v_i|^2}{\beta_i} \right) \\
& \displaystyle \leq\sum_{k=1}^{\infty} M_2\alpha_k \left(\sum_{i=1}^{\infty} \frac{|c_{ki}||v_i|^2}{\beta_i}\right) \\
&\displaystyle =M_2 \sum_{i=1}^{\infty}  \frac{|v_i|^2}{\beta_i}\left(\sum_{k=1}^{\infty}|c_{ki}|\alpha_k\right)\\
&\displaystyle \leq M_1M_2 \sum_{i=1}^{\infty}  |v_i|^2
=M_1M_2 \|v\|^2.
\end{array}
\end{equation*}
That is, $D(C)=\A$ and $C$ is a bounded linear operator, with $\|C\|\leq (M_1 M_2)^{1/2}$. 
\end{proof}


\section{The connection with discrete-time Markov chains}

In order to illustrate the applicability of Definition \ref{def:evolalg} we recover our motivation, the connection with discrete-time Markov chains. In what follows, the structure constants will be interpreted as probabilities so we assume $\mathbb{K}=\mathbb{R}$.

\begin{theorem}\label{theo:mc}
Let $\{X_n\}_{n\geq 0}$ be an homogeneous discrete-time Markov chain with state space $\mathcal{X}=\{x_i\}_{i\in \mathbb{N}}$ and transition probabilities given by $\nonumber p_{ik},$ for $i,k\in \mathbb{N}$. If  $\mathcal{A}_{\mathcal{X}}$ is a separable Hilbert space with an orthonormal basis $\{e_i\}_{i\in\mathbb{N}}$, then the structural constants $\{c_{ki}\}_{i,k\in\mathbb{N}}$, such that $c_{ki}:=p_{ik}$ for any $i,k\in\mathbb{N}$, define a Hilbert evolution algebra structure in $\mathcal{A}_{\mathcal{X}}$, called a Markov Hilbert evolution algebra. Moreover, suppose there exists $\alpha_k, \beta_i>0$, $i,k\in \mathbb{N}$ and $M_1, M_2>0$ such that
\begin{equation}\label{eq:schur test for Markov}
\begin{array}{l}
\displaystyle\sum_{k=1}^{\infty}p_{ik}\alpha_k\leq M_1\beta_i, \quad\text{ for all } i\in \mathbb{N}, \\[1.1ex]
\displaystyle\sum_{i=1}^{\infty}p_{ik}\beta_i\leq M_2\alpha_k,
\quad\text{ for all } k\in \mathbb{N}.
\end{array}
\end{equation}
Then the evolution operator $C:\mathcal{A}_{\mathcal{X}}\longrightarrow \mathcal{A}_{\mathcal{X}}$ is a bounded linear operator
and satisfies
\begin{equation}\label{eq:EOproof}
C^n(e_i)=\sum_{k=1}^{\infty} p_{ik}^{(n)} e_k,     
\end{equation}
where
$$p_{ik}^{(n)}:=\mathbb{P}(X_n=x_k | X_0 = x_i).$$
\end{theorem}

\begin{proof}
 The first part of the proof is a direct consequence of Corollary \ref{corol:manageable condition}. Indeed, consider a separable Hilbert space $\mathcal{A}_{\mathcal{X}}$ for which we can identify an orthonormal basis $\{e_i\}_{i\in\mathbb{N}}$. Then, if we consider the constants $\{c_{ik}\}_{i,k\in\mathbb{N}}$ as the transition probabilities; i.e., $c_{ki}=p_{ik}$, then condition \eqref{eq:condition HEA02} holds because for any $i\in \mathbb{N}$ we have $\sum_{k=1}^{\infty}|c_{ki}|^2 \leq \sum_{k=1}^{\infty}c_{ki} =1.$

Let $C:D(C)\longrightarrow \A_{\mathcal{X}}$ the operator  defined by \eqref{eq:EOp00}. Note that the conditions \eqref{eq:schur test for Markov} are just \eqref{eq:schur test} written for $p_{ik}=c_{ki}$. Thus, we have $D(C)=\A_{\mathcal{X}}$, the continuity of $C$ and $\|C\|\leq M^{1/2}$. Now, \eqref{eq:EOproof} can be proved by induction in $n$, by noting that 
\begin{equation}\label{eq:EO1}
C\left(C^{n-1}\left(e_i\right)\right)=C\left(\sum_{k=1}^{\infty} p_{ik}^{(n-1)} e_k\right)=\sum_{j=1}^{\infty}\left(\sum_{k=1}^{\infty} p^{(n-1)}_{ik} c_{jk}\right)e_j=\sum_{k=1}^{n}p^{(n)}_{ij}e_j,
\end{equation}
where, since $p^{(n-1)}_{ik} c_{jk}=p^{(n-1)}_{ik}p_{kj}$, the last equality of \eqref{eq:EO1} is due to the Chapman-Kolmogorov Equations (see for example \cite[Section 4.2]{ross}), which guarantee that
$$p_{ij}^{(n_1+n_2)}=\sum_{k=1}^{\infty} p^{(n_1)}_{ik} p^{(n_2)}_{kj},$$
for any $i,j,n_1,n_2\in\mathbb{N}$.
\end{proof}

Note that the  previous theorem ensures that each Markov chain with state space $\mathcal{X}=\{x_i\}_{i\in \mathbb{N}}$ induces a Markov Hilbert evolution algebra on every Hilbert space $\mathcal{A}_{\mathcal{X}}$ associated to it. On the other hand, the condition \eqref{eq:schur test for Markov} for the continuity of evolution operator is difficult to check, thus we present a particular case, which is simpler to verify.


\begin{corollary}
Let $\mathcal{A}_{\mathcal{X}}$ be a Markov Hilbert evolution algebra and $M>0$ such that  
\begin{equation}\label{eq:condition on pik}
\displaystyle\sum_{i=1}^{\infty}p_{ik}\leq M \quad\text{ for all } k\in \mathbb{N},
\end{equation}
then the evolution operator $C:\mathcal{A}_{\mathcal{X}}\longrightarrow \mathcal{A}_{\mathcal{X}}$ is  bounded 
and satisfies \eqref{eq:EOproof}.
\end{corollary}
\begin{proof}
We want to check that the equations \eqref{eq:schur test for Markov} are satisfied. Note that, since $\sum_{k=1}^{\infty}p_{ik} =1$, condition \eqref{eq:condition on pik}, allow us to use $\alpha_i=\beta_k=1$ for all $i,k\in\mathbb{N}$ and $M_1=M_2=1$. Thus, we can apply Theorem \ref{theo:mc}.
\end{proof}

\begin{example}
Let us consider the Example \ref{ex:mc} again. Let $\mathcal{X}=\{x_i\}_{i\in \mathbb{N}\cup\{0\}}$ be the state space and let the transition probabilities given by $p_{0i}=p_i>0$, for any $i\in\mathbb{N}$, where $\sum_{i=1}^{\infty}p_i=1$, and $p_{i (i-1)}=1$ for any $i\in\mathbb{N}$. By the Theorem \ref{theo:mc}, we have an induced Markov Hilbert evolution algebra $\A_{\mathcal{X}}$  by taking $c_{ki}=p_{ik}$, for $i,k\in \mathcal{X}$. Moreover, note that 
\begin{equation*}
\sum_{i=1}^{\infty} p_{i0}=p_{1 0} =1,
\end{equation*}
and, for any $k\in\mathbb{N}$, we have
\begin{equation*}
\sum_{i=1}^{\infty} p_{ik}=p_{(k+1) k} + p_{0 k}=1+p_{0 k}\leq 2.
\end{equation*}
Hence the equation \eqref{eq:condition on pik} is satisfied, implying that the evolution operator $C$ is bounded with $D(C)=\A_{\mathcal{X}}$ and $\|C\|\leq 2$.
\end{example}

\begin{example}
Consider a branching process with offspring distribution given by $(p_i)_{i\geq 0}$, with $p_0\in (0,1)$. That is, consider the discrete-time Markov chain $(Z_n)_{n\geq 0}$ such that $Z_0 =1$, and 
\begin{equation}\label{eq:BP0}
    Z_{n+1}=\sum_{i=1}^{Z_n}X_i,
    \end{equation}
where $X_1,X_2,\ldots$ are independent and identically distributed random variables with a common law $\mathbb{P}(X_1=i)=p_i$, for $i\in\mathbb{N}\cup\{0\}$. This is another example of stochastic process such that, depending of the offspring law, does not determine an evolution algebra according to Definition \ref{def:evolalgTian}. However, by the Theorem \ref{theo:mc}, we have an induced Markov Hilbert evolution algebra $\A$ by taking $c_{ki}=p_{ik}$, for $i,k\in \mathbb{N}\cup \{0\}$. Moreover, we can check equation \eqref{eq:condition on pik}. Let us consider first $k=0$. Note that,
\begin{equation}\label{eq:BP1}
\begin{array}{ccl}
\displaystyle\sum_{i=1}^{\infty} p_{i0}&=& \displaystyle\sum_{i=1}^{\infty} \mathbb{P}(Z_{n+1}=0|Z_n=i)\\[.5cm] 
&=&\displaystyle \sum_{i=1}^{\infty} \mathbb{P}\left(\sum_{\ell=1}^{i}X_{\ell}=0\right)\\[.5cm]
&=&\displaystyle \sum_{i=1}^{\infty} \mathbb{P}\left(\bigcap_{\ell=1}^{i}\{X_{\ell}=0\}\right)\\[.5cm]
&=&\displaystyle\sum_{i=1}^{\infty} p_0^{i}\\[.5cm]
&=&\displaystyle\frac{p_0}{1-p_0}.
\end{array}
\end{equation}
The second line in \eqref{eq:BP1} is due to \eqref{eq:BP0} and the independence between the $X_i$'s and $Z_n$. The fourth line is a consequence of the independence of the $X_i$'s. Now, for any $k\in \mathbb{N}$ note that, similarly to the first steps in \eqref{eq:BP1}, we have:
\begin{equation}\label{eq:BP2}
  \displaystyle   \sum_{i=1}^{\infty} p_{ik}  =\displaystyle \sum_{i=1}^{\infty} \mathbb{P}\left(\sum_{\ell=1}^{i}X_{\ell}=k\right).
 \end{equation}
Moreover, $ \big\{\sum_{\ell=1}^{i}X_{\ell}=k\big\}\subset \big\{\sum_{\ell=1}^{i}X_{\ell}\geq 1\big\}$ and, if we consider $s\in(0,1)$, note that the event $\big\{\sum_{\ell=1}^{i}X_{\ell}\geq 1\big\}$ is equivalent to the event $\big\{s^{\sum_{\ell=1}^{i}X_{\ell}}\geq s\big\}$. Thus, \eqref{eq:BP2} and the previous comments imply, by Markov's inequality: 
 \begin{equation*}
\displaystyle   \sum_{i=1}^{\infty} p_{ik}   \leq \displaystyle \sum_{i=1}^{\infty} s^{-1} \mathbb{E}\left(s^{\sum_{\ell=1}^{i}X_{\ell}}\right).
    \end{equation*}
Since the $X_i$'s are independent and identically distributed random variables, if we denote by $\varphi(s)$ the common probability generating function, we have
$$\mathbb{E}\left(s^{\sum_{\ell=1}^{i}X_{\ell}}\right)=\mathbb{E}\left(s^{X_1}\right)^i = \varphi(s)^i.$$
Therefore,
\begin{equation}\label{eq:BP3}
 \displaystyle   \sum_{i=1}^{\infty} p_{ik} \leq s^{-1} \displaystyle \sum_{i=1}^{\infty} \varphi(s)^i,
 \end{equation}
where $\varphi(s)\in [p_0,1)$ provided $s\in (0,1)$. Since $s$ can be arbitrarily chosen in $(0,1)$, take $s=1/2$, and note that from \eqref{eq:BP1} and \eqref{eq:BP3} we conclude 
$$ \displaystyle   \sum_{i=1}^{\infty} p_{ik} \leq \max\left\{\displaystyle\frac{p_0}{1-p_0},\frac{2\varphi(1/2)}{1-\varphi(1/2)}\right\},$$
for all $k\in\mathbb{N}\cup\{0\}$. Hence the equation \eqref{eq:condition on pik} is satisfied, implying that the evolution operator $C:\A_{\mathcal{X}}\longrightarrow \A_{\mathcal{X}}$ is bounded.

\end{example}

Theorem \ref{theo:mc} gains in interest if we realize that it includes as corollaries \cite[Theorem 16]{tian} and \cite[Lemma 4]{tian}. In fact, our result covers all the discrete-time Markov chains with a finite state space, and a wide range of discrete-time Markov chains with infinite, but countable, state space. It is worth pointing out that still there exist some Markov chains for which \eqref{eq:schur test for Markov} does not hold, as we illustrate in the following example.

\begin{example}\label{exe:house} [The house-of-cards Markov chain] Let $\{X_n\}_{n\geq 0}$ be the Markov chain with state space given by $\mathcal{X}=\mathbb{N}\cup\{0\}$ and transition probabilities given by $p_{i 0}=p_i>0$ and $p_{i(i+1)}=1-p_i$, for any $i\in\mathbb{N}$, and $p_{00}=p_0=1-p_{01}$. See Figure \ref{exe:house} for an illustration of the transitions of such a Markov chain. In words, from any state $i\neq 0$ the process jumps to $0$ with probability $p_i$ or it jumps to state $i+1$ with probability $1-p_i$. This model is known as the house-of-cards Markov chain. Notice that to state $0$ we can go from infinitely many states with positive probability. For this chain we can find examples for which \eqref{eq:schur test for Markov} holds, or not.
\begin{enumerate}
\item[i.] If we let $\sum_{i=0}^{\infty}p_i =1$, then $\alpha_i=\beta_i=1$ for any $i$, and $M_1=M_2=1$ is enough to satisfy \eqref{eq:schur test for Markov}. In fact, we would have for any $i\in\mathbb{N}\cup\{0\}$
$$\sum_{k=0}^{\infty}p_{ik}\alpha_k=p_{i0}\alpha_0 + p_{i(i+1)}\alpha_{i+1}=p_i + (1-p_i)=1.$$
Moreover, for any $k\in\mathbb{N}$, we would have
$$\sum_{i=0}^{\infty}p_{ik}\beta_i = p_{(k-1)k}\beta_{k-1}=(1-p_{k-1})<1.$$
while for $k=0$ we would have
$$\sum_{i=0}^{\infty}p_{i0}\beta_i = \sum_{i=0}^{\infty}p_{i}=1,$$
which completes the verification of \eqref{eq:schur test for Markov}.

\smallskip
\item[ii.] Let $p_i=p$ for all $i\in \mathbb{N}\cup\{0\}$, and suppose that \eqref{eq:schur test for Markov} holds. Then, for any $i\in\mathbb{N}\cup\{0\}$ we have
$$\sum_{k=0}^{\infty}p_{ik}\alpha_k =p_{i0}\alpha_0 + p_{i(i+1)}\alpha_{i+1}=p\alpha_0 + (1-p)\alpha_{i+1},$$
which implies that
\begin{equation}\label{eq:house1}
p\,\alpha_0 + (1-p)\alpha_{i+1}\leq M_1\beta_i
\end{equation}
for all $i\in\mathbb{N}\cup\{0\}$. On the other hand, 
$$\sum_{i=0}^{\infty}p_{i0}\beta_i = p \sum_{i=0}^{\infty}\beta_i,$$
which implies by \eqref{eq:schur test for Markov} that
\begin{equation}\label{eq:house1}
\sum_{i=0}^{\infty}\beta_i\leq M_2\alpha_0 <\infty.
\end{equation}
Then $\lim_{i\to \infty}\beta_i =0$, and this in turns implies from \eqref{eq:house1} that $p\,\alpha_0=0$, which is a contradiction.
\end{enumerate}
\end{example}

\begin{figure}
    \centering
    \begin{tikzpicture}[scale=0.8, every node/.style={scale=0.8}]

\draw [thick] (0,0) circle (7pt);
\draw (0,0) node[font=\footnotesize] {$0$};
\draw [thick] (2,0) circle (7pt);
\draw (2,0) node[font=\footnotesize] {$1$};
\draw [thick] (4,0) circle (7pt);
\draw (4,0) node[font=\footnotesize] {$2$};
\draw [thick] (6,0) circle (7pt);
\draw (6,0) node[font=\footnotesize] {$3$};
\draw [thick] (8,0) circle (7pt);
\draw (8,0) node[font=\footnotesize] {$4$};


\draw [thick, directed] (0,0.25) to [out=140,in=220,looseness=8] (0,-0.25);
\draw [thick, reverse directed] (0,0.25) to [bend left=45] (2,0.25);
\draw [thick, directed] (0,-0.25) to [bend right=45] (2,-0.25);
\draw [thick, directed] (2,-0.25) to [bend right=45] (4,-0.25);
\draw [thick, directed] (4,-0.25) to [bend right=45] (6,-0.25);
\draw [thick, directed] (6,-0.25) to [bend right=45] (8,-0.25);
\draw [dashed, gray, directed] (8,-0.25) to [bend right=90] (10,-0.5);
\draw [thick, reverse directed] (0,0.25) to [bend left=60] (4,0.25);
\draw [thick, reverse directed] (0,0.25) to [bend left=75] (6,0.25);
\draw [thick, reverse directed] (0,0.25) to [bend left=90] (8,0.25);
\draw [dashed, gray, reverse directed] (0,0.25) to [bend left=90] (10,1.25);


\draw (-1.2,0) node[font=\footnotesize] {$p_0$};
\draw (1.8,0.8) node[font=\footnotesize] {$p_1$};
\draw (1.3,-1.1) node[font=\footnotesize] {$1-p_0$};
\draw (3.3,-1.1) node[font=\footnotesize] {$1-p_1$};
\draw (5.3,-1.1) node[font=\footnotesize] {$1-p_2$};
\draw (7.3,-1.1) node[font=\footnotesize] {$1-p_3$};
\draw (3.4,1.3) node[font=\footnotesize] {$p_2$};
\draw (4.7,2) node[font=\footnotesize] {$p_3$};
\draw (6.1,2.6) node[font=\footnotesize] {$p_4$};
\draw (9,0) node[font=\footnotesize] {$\cdots$};

\end{tikzpicture}
    \caption{Graphical representation of the Markov chain of Example \ref{exe:house}.}
    \label{fig:mc}
\end{figure}
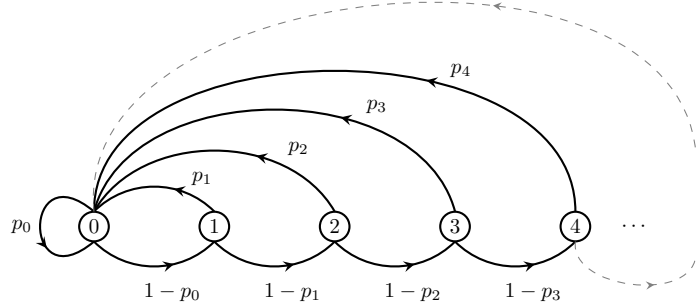

As pointed out in \cite[Lemma 4 of Chapter 4]{tian}, and now extended to the cases where \eqref{eq:schur test for Markov} holds, the evolution operator can be used in the context of evolution algebras as the transition matrix, whose entries are transition probabilities, is used in the context of Markov chains. 
This is the spirit of \eqref{eq:EOproof}, which can be easily extended to any $v\in \mathcal{A}_{\mathcal{X}}$ such that $v=\sum_{i=1}^{\infty} \alpha_i e_i$, with $\{\alpha_i\}_{i\in\mathbb{N}}$ being a probability distribution on $\mathcal{X}$; i.e., $\alpha_i\in[0,1]$ for any $i\in\mathbb{N}$ and $\sum_{i\in\mathbb{N}} \alpha_i=1$. In other words, $v$ is a (possibly infinite) convex combination of points of the orthonormal basis $\{e_i\}_{i\in\mathbb{N}}$. Let us denote by $\overline{\conv}(A)$ the closed convex hull of the set $A$, that is, the closure of the convex hull $\conv(A)$. 

\begin{corollary}\label{cor:last}
Consider a Markov Hilbert evolution algebra $\mathcal{A}_{\mathcal{X}}$ satisfying the equations \eqref{eq:schur test for Markov}
and let $v\in \overline{\conv}\left(\{e_i\}_{i\in\mathbb{N}}\right)$ such that $v=\sum_{i=1}^{\infty} \alpha_i e_i$. Then
$$C^{n}(v)=\sum_{i=1}^{\infty}\alpha^n_ i(v) e_i,$$
where $\alpha^n_i(v)=\mathbb{P}(X_n=x_i)$ provided $\mathbb{P}(X_0=x_k)=\alpha_{k}$, for $k\in\mathbb{N}$.
\end{corollary}

\begin{proof}
Let $v\in \overline{\conv}\left(\{e_i\}_{i\in\mathbb{N}}\right)$ such that $v=\sum_{i=1}^{\infty} \alpha_i e_i$, and assume that $\mathbb{P}(X_0=x_k)=\alpha_{k}$, for $k\in\mathbb{N}$. The law of total probability implies
$$\mathbb{P}(X_n=x_i)=\sum_{k=1}^{\infty}\mathbb{P}(X_n=x_i|X_0=x_k)\mathbb{P}(X_0=x_k)=\sum_{k=1}^{\infty}p_{ki}^{(n)}\alpha_k.$$
The proof is finished if we let $\alpha_{i}^n(v):=\mathbb{P}(X_n=x_i)$ and we realize that \eqref{eq:EOproof} implies
$$C^n(v)=\sum_{k=1}^{\infty} \alpha_k\, C^n(e_k)= \sum_{k=1}^{\infty} \alpha_k\left\{\sum_{i=1}^{\infty}p_{ki}^{(n)}e_i\right\}=\sum_{i=1}^{\infty}\left\{\sum_{k=1}^{\infty}\alpha_k\,p_{ki}^{(n)}\right\}e_i.$$
\end{proof}

Let us finish with a comment about this connection with Markov chains. In words, the previous results claim that it is possible to model some random phenomena with an approach of Hilbert evolution algebras. When one uses Markov chains the first step is to identify the state space of the process, the second one is to determine the transition probabilities. The conclusion of this section is that if we associate to each possible state of the process a generator of the algebraic structure, then the whole evolution of the process can be observed through consecutive applications of the evolution operator,  provided \eqref{eq:schur test for Markov}, or \eqref{eq:condition on pik} holds. Indeed, Corollary \ref{cor:last} can be applied by assuming that the application of $C$ to $v\in \overline{\conv}\left(\{e_i\}_{i\in\mathbb{N}}\right)$ represents that the process starts from the state represented by $e_i$ with probability $\alpha_i$. Then, the application of the evolution operator $n$ times allows to discover with which probability the process will be in a given state at time $n$. We point out that our extension of the concept of evolution algebra allow to advance in the analysis proposed in \cite[Chapter 4]{tian}, extending it to a wide class of infinite-dimensional Markov chains.

\smallskip
\section{Acknowledgments}
Part of this work was carried out during a visit of P.C. at the Universidade Federal de Pernambuco (UFPE); and visits of P.M.R. at the Universidade Federal do ABC (UFABC) and at the Universidad Nacional de la Patagonia ``San Juan Bosco'' (UNPSJB). The visit at UNPSJB was during the realization of the School EMALCA 2019. The authors are grateful with these institutions, and with the organizers of the School, for their hospitality and support. Part of this work has been supported by Fundação de Amparo à Pesquisa do Estado de São Paulo - FAPESP (Grant 2017/10555-0).


\bigskip


\begin{thebibliography}{99}


\bibitem{PMP}
P. Cadavid, M. L. Rodi\~no Montoya and P. M. Rodriguez, {\it The connection between evolution algebras, random walks, and graphs}.  J. Algebra Appl. {\bf 19} n.2 (2020): 2050023.

\bibitem{PMPT} P. Cadavid, M. L. Rodi\~no Montoya  and  P. M. Rodriguez, {\it On the isomorphisms between evolution algebras of random walks and graphs}. {\bf 69} n.10  (2021): 1858-1877.


\bibitem{PMP2} P. Cadavid, M. L. Rodi\~no Montoya  and  P. M. Rodriguez, {\it Characterization theorems for the space of derivations of evolution algebras associated to graphs}. Linear Multilinear Algebra. {\bf 68} n.7 (2020): 1340-1354.

\bibitem{PMPY} Y. Cabrera Casado, P. Cadavid, M.L. Rodi\~no Montoya  and  P.M. Rodriguez, {\it On the characterization of the space of derivations in evolution algebras}. Annali di Matematica. {\bf 200} (2021): 737–755.

\bibitem{casado/molina/velasco/2016}Y. Cabrera, M.  Siles and M. V. Velasco, {\it Evolution algebras of arbitrary dimension and their decomposition}. Linear Algebra Appl. {\bf 495} (2016): 122-162.


\bibitem{camacho/gomez/omirov/turdibaev/2013}
L. M. Camacho, J. R. G\'omez, B. A. Omirov and R. M. Turdibaev, {\it Some properties of evolution algebras}. Bull. Korean Math. Soc. {\bf 50} n. 5 (2013): 1481-1494. 

\bibitem{Casas/Ladra/2014}
J. M. Casas, M. Ladra, B.A. Omirov and U.A. Rozikov, {\it On Evolution Algebras}, Algebra Colloq. {\bf 21} (2014): 331.

\bibitem{Elduque/Labra/2015}
A. Elduque and A. Labra, {\it Evolution algebras and graphs}, J. Algebra Appl. {\bf 14} (2015): 1550103.

\bibitem{Elduque/Labra/2016}
A. Elduque and A. Labra, {\it On nilpotent evolution algebras},  Linear Algebra Appl. {\bf 505} (2016): 11-31.

\bibitem{Falcon/Falcon/Nunez/2017}
O.J. Falc\'on, R.M. Falc\'on and J. N\'u\~nez, {\it Classification of asexual diploid organisms by means of strongly isotopic evolution algebras defined over any field}. J. Algebra {\bf 472} (2017): 573-593.

\bibitem{Halmos} P.R. Halmos, {\it A Hilbert space problem book}, Vol. 19, Springer Science and Business Media, 2012.

\bibitem{Heil} C. Heil, {\it Introduction to Real Analysis}, Vol. 280, Springer, 2019.

\bibitem{Labra/Ladra/Rozikov/2014}
A. Labra, M. Ladra and U. A. Rozikov, {\it An evolution algebra in population genetics}, Linear Algebra Appl. {\bf 457} (2014): 348-362.


\bibitem{Mellon/Velasco} P. Mellon and V. Velasco, {\it  Analytic aspects of evolution algebras}. Banach J. Math. Anal.  {\bf  13} n.1 (2019): 113-132.

\bibitem{reis}
T. Reis and P. Cadavid, {\it Derivations of evolution algebras associated to graphs over a field of any characteristic}. Linear Multilinear Algebra. DOI: 10.1080/03081087.2020.1818673

\bibitem{ross}
S. M. Ross, {\it Introduction to Probability Models}, 10th Ed., Academic Press, 2010. 

\bibitem{tian}
J. P. Tian, {\it Evolution algebras and their applications}, Springer-Verlag Berlin Heidelberg, 2008.

\bibitem{tv}
J. P. Tian and P. Vojtechovsky, {\it Mathematical concepts of evolution algebras in non-Mendelian genetics}. Quasigroups Related Systems {\bf 1} n. 14 (2006):111-122.
\end{thebibliography}
\end{document}